\documentclass[a4paper.12pt]{article}

\usepackage[top=2cm, bottom=2cm, left=2cm, right=2cm]{geometry}

\usepackage[active]{srcltx}

\usepackage[latin1]{inputenc}
\usepackage{graphicx}
\usepackage{latexsym, amsmath, amsthm, amsfonts, amssymb, txfonts, pxfonts, wasysym}
\usepackage{latexsym}
\usepackage{algorithm}
\usepackage{caption}
\usepackage[dvips]{color}
\usepackage{xy}
\usepackage{wrapfig}
\usepackage{dsfont}
\input xy
\xyoption{all}

\theoremstyle{remark}

\newtheorem*{proofpar}{\hspace{7 mm}Proof}
\newtheorem{definition}[subsection]{\hspace{7 mm}Definition}

\theoremstyle{theorem}

\newtheorem{theorem}[subsection]{\hspace{7 mm}Theorem}

\newtheorem{corollary}[subsection]{\hspace{7 mm}Corollary}
\newtheorem{lemma}[subsection]{\hspace{7 mm}Lemma}

\renewcommand\proofname[1]{\hspace{7 mm}Proof}

\title{A note on the combinatorial structure of finite and locally finite simplicial complexes of nonpositive curvature}

\author{
Djordje Barali\'{c}\\
Mathematical Institute, Serbian Academy of Sciences and Arts,\\
Kneza Mihaila $36$, p.p. $367, 11001$ Belgrade, Serbia\\
E-mail address: djbaralic@mi.sanu.ac.rs\\
\\
Ioana-Claudia Laz\u{a}r\\
'Politehnica' University of Timi\c{s}oara, Dept. of Mathematics,\\
Victoriei Square $2$, $300006$-Timi\c{s}oara, Romania\\
E-mail address: ioana.lazar@mat.upt.ro}

\date{}

\usepackage{pslatex}

\begin{document}

\maketitle

\begin{abstract}

We investigate the collapsibility of systolic finite simplicial complexes of arbitrary dimension.  The main tool we use in the proof is discrete Morse theory. We shall consider a convex subcomplex of the complex and project any simplex of the complex onto a ball around this convex subcomplex. These projections will induce a convenient gradient matching on the complex.
Besides we analyze the combinatorial structure of both CAT(0) and systolic locally finite simplicial complexes of arbitrary dimensions. We will show that both such complexes possess an arborescent structure. Along the way we make use of certain well known results regarding systolic geometry.

\hspace{0 mm} \textbf{2010 Mathematics Subject Classification}:
05C99, 05C75.

\hspace{0 mm} \textbf{Keywords}: standard piecewise Euclidean metric, CAT(0) metric, local $6$-largeness, convex subcomplex, projection ray, directed geodesic, discrete vector field, gradient path, Morse matching, collapsibility,
arborescent structure.
\end{abstract}

\section*{Introduction}

In this paper we show that finite simplicial complexes satisfying a combinatorial curvature condition can be simplicially collapsed to a point, while nonpositively curved locally finite simplicial complexes have a rather simple combinatorial structure.

Curvature can be expressed both in metric and combinatorial terms. One can either refer to 'nonpositively curved' in the sense of Aleksandrov and Gromov, i.e. by comparing small triangles in the space with triangles in the Euclidean plane. Such triangles must satisfy the CAT(0) inequality (see \cite{bridson_1999} or \cite{gromov_1987}). Or, else, one can express curvature combinatorially using a condition, called local $6$-largeness, that seems a good analogue of metric nonpositive curvature (see \cite{janusz_2006}). A simplicial complex is locally $6$-large if every cycle consisting of less than $6$ edges in any of its links has some two consecutive edges contained in a $2$-simplex of this link. We call a simplicial complex systolic if it is locally $6$-large, connected and simply connected. In dimension $2$, local $6$-largeness is called the $6$-property of the simplicial complex (see \cite{corson_1998}). In dimension $2$ the two curvature conditions are equivalent for the standard piecewise Euclidean metric. In higher dimensions, however, the equivalence no longer holds (see \cite{janusz_2006}, chapter $14$ for examples). Still, one implication remains true on simplicial complexes
of dimension greater than $2$. Namely, there exists a constant $k(n) \geq 6$ such that
the standard
piecewise Euclidean metric on a $k(n)$-systolic, $n$-dimensional
simplicial complex whose simplices represent only finitely many isometry classes, is
CAT($0$) (see \cite{janusz_2006} (chapter $14$, page
$52$).

It turns out that systolic complexes have many properties similar to CAT(0) spaces. In dimension
$2$, for instance, both CAT(0) and systolic simplicial complexes (not necessarily endowed with the
standard piecewise Euclidean metric), collapse to a point (see
\cite{lazar_2010_8}, chapter $3.1$, page $36 - 48$; \cite{corson_1998}).
Besides, the weaker condition of
contractibility does characterize both nonpositively curved spaces and systolic complexes. The distance function from any point in a CAT(0) space has a unique local minimum on each convex subset of the space. Similarly, in a systolic complex the projection of any simplex in the $1$-ball around a convex subcomplex onto the subcomplex, is a single simplex.

Crowley showed in \cite{crowley_2008}, using discrete Morse theory (see
\cite{forman_1998}), that CAT(0)
standard piecewise Euclidean
complexes of dimension $3$ or less satisfying the property
that any $2$-simplex is a face of at most two $3$-simplices in the
complex, simplicially collapse to a point.
It is shown further in \cite{lazar_2012}, using again discrete Morse theory, that in dimension $2$
and $3$, under the same technical condition, systolic standard piecewise Euclidean complexes
enjoy the same property. The novelty of the result in \cite{lazar_2012} is that, in dimension $3$,
the hypothesis is no longer of metric nature.
The result in \cite{lazar_2012} in dimension $2$, however, is equivalent to Crowley's. This happens because the CAT(0)
$2$-complex in Crowley's paper is constructed by endowing it with the
standard piecewise Euclidian metric and requiring that each of its interior
vertices has degree at least $6$. The standard piecewise Euclidian
metric on the $2$-complex becomes then CAT(0).

Adiprasito and Benedetti extended Crowley's result to all dimensions (see \cite{benedetti_2012}, Theorem $3.2.1$). Namely, they proved using discrete Morse theory that every complex that is CAT(0) with a metric for which all vertex stars are convex, is collapsible. Compared to Crowley's paper, the techniques used in \cite{benedetti_2012} when applying discrete Morse theory are more advanced.

Given the similarities between the CAT(0) and systolic worlds,
one may naturally ask whether a result similar to the one in \cite{benedetti_2012} also holds on systolic complexes. An affirmative answer to this question was already given in \cite{chepoi_2009} and \cite{lazar_2013}. One aim of the paper is to give a different proof of the same result by applying discrete Morse theory.
Compared to the proof given by Adiprasito and Benedetti on CAT(0) complexes, however, the collapsibility of systolic complexes can be shown much easier. The key point is to define conveniently a so called gradient matching on the complex. Due to certain results regarding systolic geometry, this gradient matching will turn out to be a Morse matching. One of the properties of systolic complexes we shall refer to is: any two vertices in a systolic complex can be joined by a unique directed geodesic.

Knowing that systolic and CAT(0) finite simplicial complexes are both collapsible, one may naturally ask what happens in the infinite case. The equivalent notion of collapsibility is called arborescent structure.
A single result regarding the arborescent structure of simplicial complexes is known to us. Namely, in \cite{corson_1998} it is shown that any locally finite, simply connected  simplicial $2$-complex with the $6$-property is a monotone union of a sequence of collapsible subcomplexes. We show that the same holds in all dimensions. Besides, we give an analogue result on locally finite CAT(0) simplicial complexes. The results in the infinite case are consequences of the fact that both systolic and CAT(0) finite simplicial complexes endowed with standard piecewise Euclidean metrics are collapsible.

\textbf{Acknowledgements.} The first author was supported by the
Grant 174020 of the Ministry for Education and Science of the
Republic of Serbia. The second author is grateful to Professor
Louis Funar for advice and thanks the Georg-August University in
G\"ottingen for its hospitality during the summer of $2013$. The
visit has been supported by the DAAD.

\section{Preliminaries}

\subsection{CAT(0) spaces}

Let $(X,d)$ be a metric space.
A \emph{geodesic path} joining $x \in
X$ to $y \in X$ is a path $c : [a,b] \to X$ such that $c(a) = x$,
$c(b) = y$ and $d(c(t), c(t')) = |t - t'|$ for all $t, t' \in
[a,b]$. The image $\alpha$ of $c$ is called a \emph{geodesic
segment} with endpoints $x$ and $y$.

A \emph{geodesic metric space} $(X,d)$ is a metric space in which every pair of points
can be joined by a geodesic segment. We denote any geodesic segment from a point $x$ to a point $y$ in $X$,
by $[x,y]$.

A \emph{geodesic triangle} in $X$ consists of three points $p,q,r
\in X$, called \emph{vertices}, and a choice of three geodesic
segments $[p,q], [q,r], [r,p]$ joining them, called \emph{sides}.
Such a geodesic triangle is denoted by $\triangle ([p,q], [q,r],
[r,p])$ or $\triangle (p,q,r)$. If a point $x \in X$ lies in the
union of $[p,q], [q,r]$ and $[r,p]$, then we write $x \in
\triangle$. A triangle $\overline{\triangle} =
\triangle(\overline{p},\overline{p},\overline{r})$ in
$\mathds{R}^{2}$ is called a \emph{comparison triangle} for
$\triangle = \triangle (p,q,r)$ if $d(p,q) =
d_{\mathds{R}^{2}}(\overline{p},\overline{q})$, $d(q,r) =
d_{\mathds{R}^{2}}(\overline{q},\overline{r})$ and $d(r,p) =
d_{\mathds{R}^{2}}(\overline{r},\overline{p})$. A point
$\overline{x} \in [\overline{q},\overline{r}]$ is called a
\emph{comparison point} for $x \in [q,r]$ if $d(q,x) =
d_{\mathds{R}^{2}}(\overline{q},\overline{x})$.

Let $\triangle (p,q,r)$ be a geodesic
triangle in $X$. Let $\overline{\triangle}
(\overline{p},\overline{q},\overline{r}) \subset \mathds{R}^{2}$ be
a comparison triangle for $\triangle$. The metric $d$ is
\emph{CAT(0)} if for all $x,y \in \triangle$ and all comparison
points $\overline{x}, \overline{y} \in \overline{\triangle}$, the
CAT(0) inequality holds:
\begin{center}
$d(x,y) \leq d_{\mathds{R}^{2}}(\overline{x}, \overline{y})$.
\end{center}
We call $X$ a \emph{CAT(0) space} if it is a geodesic metric
space all of whose geodesic triangles satisfy the CAT(0) inequality.

We say $X$ is \emph{of curvature $\leq 0$} (or
\emph{nonpositively curved}) if it is locally a CAT(0) space, i.e.
for every $x \in X$ there exists $r_{x} > 0$ such that the ball
$B(x, r_{x}) = \{ y \in X | d(x,y) \leq r_{x} \}$, endowed with the induced metric, is a CAT(0) space.

Any complete connected metric space that is simply connected and of curvature $\leq 0$, is a CAT(0) space (for the proof see \cite{bridson_1999}, chapter II.$4$, page $194$).

The balls in a CAT(0) space are contractible; in particular they are simply
connected (for the proof see \cite{bridson_1999}, chapter II.$1$, page $160$).

Let $X$ be a connected CAT(0) space and let $x$ be a point of $X$. The function distance from $x$ has a unique local minimum on each convex subset of $X$ (see \cite{bridson_1999}, Proposition $2.4$, page $176$).

\subsection{Systolic simplicial complexes}

Let $K$ be a simplicial complex.
The underlying space $|K|$ of $K$ is the union of its simplices as topological space. We denote by $K^{(k)}$ the $k$-skeleton of $K, 0 \leq k < \dim K$.

We call $K$ {\it locally finite} if for any simplex $\sigma$ of $K$ and for any point $x \in \sigma$ whereas $\sigma$ is a simplex of $K$, some neighborhood of $x$
meets only finitely many simplices $\tau$ in $K$.

Let $K$ be a finite, connected simplicial complex endowed with the standard
piecewise Euclidean metric. We define the {\it standard piecewise Euclidean metric}
on $|K|$ by taking the distance between any two points $x,y$ in $|K|$ to be the
infimum over all paths in $|K|$ from $x$ to $y$. Each simplex of $K$ is isometric with a regular Euclidean simplex of the same dimension with side lengths equal $1$.

All simplicial complexes in this paper are finite dimensional.
The topology of any simplicial complex in this paper is induced by the standard piecewise Euclidean metric.

The \emph{join} $\sigma \ast \tau$ of two simplices $\sigma$ and $\tau$ of $K$ is a simplex whose vertices are the vertices of $\sigma$ plus the vertices of $\tau$.
The \emph{link} of $K$ at $\sigma$, denoted $\rm {Lk}(\sigma, K)$, is
the subcomplex of $K$ consisting of all simplices of $K$ which are disjoint
from $\sigma$ and which, together with $\sigma$, span a simplex of
$K$. The (closed) \emph{star} of $\sigma$ in $K$, denoted
$\rm {St}(\sigma, K)$, is the union of all simplices of $K$ that contain
$\sigma$. The star $\rm {St}(\sigma, K)$ is the join of $\sigma$
and the link $\rm {Lk}(\sigma, K)$. A subcomplex $L$ in
$K$ is called \emph{full} (in $K$) if any simplex of $K$ spanned by a set of
vertices in $L$, is a simplex of $L$.  We call $K$ {\it flag} if any finite set of vertices, which are pairwise connected by edges of $K$, spans a simplex of $K$.

A {\it cycle} $\gamma$ in $K$ is a subcomplex of $K$ isomorphic to a triangulation of $S^{1}$. The length of $\gamma$ (denoted by $|\gamma|$) is the number of edges in $\gamma$.
A \emph{full cycle} in $K$ is
a cycle that is full as subcomplex of $K$.
We define the \emph{systole} of $K$
by \begin{center}$sys(K) = \min \{ |\alpha|: \alpha $ is a full cycle in $K \}$.\end{center}

Given a natural $k \geq 4,$ we call $K$ $k$\emph{-large} if
$sys(K) \geq k$ and $sys(\rm {Lk}(\sigma, K)) \geq k$ for each simplex
$\sigma$ of $K$. We call $K$ \emph{locally} $k$\emph{-large} if the
star of every simplex of $K$ is $k$-large. We call $K$
$k$-\emph{systolic} if it is connected, simply connected and locally
$k$-large. We abbreviate $6$-systolic to systolic.

A simplicial complex if $4$-large if and only if it is flag (see \cite{janusz_2006}, Fact $1.2(3)$, page $9$). In particular, systolic simplicial complexes are flag.

Given a subcomplex $Q$ of $K$, a {\it cycle in the pair} $(K,Q)$ is a polygonal path $\gamma$ in the $1$-skeleton of $K$ with endpoints contained in $Q$ and without self-intersections, except a possible coincidence of the endpoints. Such a cycle $\gamma$ is {\it full} in $(K,Q)$ if its simplicial span in $K$ is contained in the union $\gamma \cup Q$. A subcomplex $Q$ in $K$ is $3$-{\it convex} if $Q$ is full in $K$ and every full cycle in $(K,Q)$ of length less than $3$ is contained in $Q$. Any simplex in a flag simplicial complex is a $3$-convex subcomplex of the complex (see \cite{janusz_2006}, Example $3.1$, page $17$).

A subcomplex $Q$ is {\it locally} $3$-{\it convex} in $K$ if for every nonempty simplex $\sigma$ of $Q$, the link $Q_{\sigma}$ is $3$-convex in the link $K_{\sigma}$. Any $3$-convex subcomplex in a flag complex $K$ is a locally $3$-convex subcomplex in $K$ (see \cite{janusz_2006}, Fact $3.3$, page $18$).

We call a subcomplex $Q$ in a systolic simplicial complex $K$ {\it convex} if it is connected and locally $3$-convex.

The {\it combinatorial
distance} $d_{c}(v,v')$ between two vertices $v$ to $v'$ of a simplicial complex $K$ is the length of the edge-path joining $v$ to $v'$ in $K$ which has, among all edge-paths in $K$ joining $v$ to $v'$, the shortest length.

For a subcomplex $Q$ of a simplicial complex $K$, denote by $N_{K}(Q)$ the subcomplex of $K$ being the union of all (closed) simplices that intersect $Q$.

For a convex subcomplex $Q$ in a systolic complex $K$, define a system $B_{n} = B_{n}Q = \{ x \in K^{(0)} | d_{c}(x,v) \leq n, \forall v \in Q^{(0)} \}$ of combinatorial balls in $K$ of radii $n$ centered at $Q$ as $B_{0}:=Q$ and $B_{n+1}:=N_{K}(B_{n})$ for $n \geq 0$.

For $n \geq 1,$ the {\it sphere} of radius $n$ centered at a convex subcomplex $Q$ is the full subcomplex $S_{n}Q$ in $K$ spanned by the vertices at combinatorial distance $n$ from $Q$.

\begin{lemma}\label{1.2.1}
Let $Q$ be a convex subcomplex of a systolic complex $K$ and let $n \geq 1$ be a natural number. Then:
\begin{enumerate}
\item $B_{n}Q$ is the full subcomplex of $K$ spanned by the set of all vertices of $K$ at combinatorial distance $\leq n$ from $Q$;
\item $S_{n}Q \subset B_{n}Q$;
\item $S_{n}Q$ is equal to the union of those simplices in the ball $B_{n}Q$ which are disjoint with $B_{n-1}Q$ (for the proof see \cite{janusz_2006}, Lemma $7.6$, page $34$).
\end{enumerate}
\end{lemma}

\begin{lemma}\label{1.2.2}
Any convex subcomplex of a systolic complex is a $3$-convex subcomplex of the complex (see \cite{janusz_2006}, Lemma $7.2$, page $32$).
\end{lemma}

\begin{lemma}\label{1.2.3}
For any convex subcomplex $Q$ in a systolic complex $K$ and for
any simplex $\sigma \subset N_{K}(Q)$ disjoint from $Q$, the
intersection $Q \cap \rm{St}(\sigma, K)$ is a single simplex of
$Q$ (for the proof see \cite{janusz_2006}, chapter $7$, page
$34$).
\end{lemma}

We introduce and study further projections of systolic complexes onto convex subcomplexes.

Let $K'$ denote the first barycentric subdivision of a simplicial complex $K$. For a simplex $\sigma$ of $K$, let $b_{\sigma}$ denote the barycenter of $\sigma$ which is a vertex of $K'$.

Let $Q$ be a convex subcomplex of a systolic complex $K$. We call
an {\it elementary projection} a simplicial map $\pi_{Q} :
(B_{1}Q)' \rightarrow Q'$ between the first barycentric
subdivision of these subcomplexes defined as follows: if $\sigma
\cap Q \neq \varnothing$, we put $\pi_{Q}(b_{\sigma}) = b_{\sigma
\cap Q}$; if $\sigma \cap Q = \varnothing,$ then we put
$\pi_{Q}(b_{\sigma}) = b_{\tau},$ whereas $\tau = \rm{St} (\sigma,
K) \cap Q$. Afterwards we extend simplicially. According to the
previous two lemmas, the simplicial map $\pi_{Q}$ is well defined.

We denote by $P^{n}_{Q} : (B_{n}Q)' \rightarrow Q'$ the composition map $\pi_{B_{n-1}Q} \circ \pi_{B_{n-2}Q} \circ ... \circ \pi_{B_{1}Q} \circ \pi_{Q}$. We denote by $P_{Q} : X' \rightarrow Q'$ the union $\cup_{n}P^{n}_{Q}$. We call $P_{Q}$ the {\it projection} to $Q$.

We will use the following variants of projection maps between face posets rather than barycentric subdivisions. If $Q$ is a convex subcomplex of $K$ and $\sigma$ is a simplex of $K$, then put $\overline{P}_{Q}(\sigma) = \tau$ if and only if $P_{Q}(\sigma) = \tau$ and put $\overline{\pi}_{Q}(\sigma) = \tau$ if and only if $\pi_{Q}(\sigma) = \tau$.

Let $Q$ be a convex subcomplex of $K$ and let $\sigma$ be a simplex of $S_{n}(Q)$. The {\it projection ray} from $\sigma$ to $Q$ is the sequence $\sigma = \sigma_{0}, \sigma_{1}, ..., \sigma_{n}$ of simplices in $K$ given by $\sigma_{k} = \overline{\pi}_{B_{n-k+1}Q}\sigma_{k-1}$ for $1 \leq k \leq n$. Equivalently, this sequence is given by $\sigma_{k} = \overline{P}_{B_{n-k}Q}(\sigma_{0})$.

Any two consecutive simplices $\sigma_{k}, \sigma_{k+1}$ in a projection ray are disjoint and span a simplex of $K$ (see \cite{janusz_2006}, Fact $8.3.1$, page $37$).

If $\sigma_{k}$ and $\sigma_{m}$ are simplices in a projection ray then for any vertices $v \in \sigma_{k}$ and $w \in \sigma_{m}$ we have $d_{c}(v,w) = |k-m|$ (see \cite{janusz_2006}, Fact $8.3.2$, page $37$).

A projection ray in a systolic complex is uniquely determined by its initial and final simplex (see \cite{janusz_2006}, Corollary $8.5$, page $37$).

A sequence of simplices $(\sigma_{n})$ in a systolic simplicial complex $K$ is a {\it directed geodesic} if it satisfies the following properties:
\begin{enumerate}
\item any two consecutive simplices $\sigma_{i}, \sigma_{i+1}$ in the sequence are disjoint and span a simplex of $K$;
\item for any three consecutive simplices $\sigma_{i}, \sigma_{i+1}, \sigma_{i+2}$ in the sequence we have $\rm{St} (\sigma_{i},K) \cap B_{1}(\sigma_{i+2},K) = \sigma_{i+1}$.
\end{enumerate}

Each projection ray in a systolic simplicial complex is a directed geodesic (for the proof see \cite{janusz_2006}, chapter $9$, page $39$).

A directed geodesic $\sigma_{0}, ..., \sigma_{n}$ in a systolic complex is a projection ray on its final simplex $\sigma_{n}$ (for the proof see \cite{janusz_2006}, chapter $9$, page $39$).

The above results imply that the sets of finite directed geodesics and of projection rays coincide.

Given two vertices $v,w$ in a systolic complex, there is exacty one directed geodesic from $v$ to $w$ (see \cite{janusz_2006}, Corollay $9.7$, page $39$). There is another unique directed geodesic from $w$ to $v$.

Let $v,w$ be two vertices in a systolic complex such that $d_{c}(v,w) = n$. Then the directed geodesic from $v$ to $w$ consists of $n+1$ simplices (see \cite{janusz_2006}, Corollay $9.8$, page $39$).

\subsection{Discrete Morse theory}

The main tool we shall use in the proof is discrete Morse theory. In
this section we introduce the main notions in discrete Morse theory
and we give a few basic results regarding these notions.

Let $(K,\subseteq)$ denote the face poset of a simplicial complex $K$, i.e. the set of nonempty simplices of $K$ ordered with respect to inclusion. Let $(\mathbf{R},\leq)$ denote the poset of real numbers with the usual ordering. A {\it discrete Morse function} is an order-preserving map $f:K\rightarrow \mathbf{R}$ that assigns to each simplex in the complex a real number such that the preimage $f^{-1}(r)$ of any number $r$ consists of at most $2$ elements. The function assigns the same value to two distinct simplices in the complex if one of the simplices is contained in the other. We call a simplex of $K$ {\it critical} if it is a simplex at which $f$ is strictly increasing.

The function $f$ induces a perfect matching on the non-critical simplices: we match two simplices if $f$ assigns them the same real number. We call this matching a {\it Morse matching} and we represent it by a system of arrows: for any two simplices $\sigma, \tau$ such that $\sigma \subsetneq \tau$ and $f(\sigma) = f(\tau)$, one draws an arrow from the barycenter of $\sigma$ to the barycenter of $\tau$.
We will represent a Morse matching by its associated partial function $\theta : K \rightarrow K$ which is defined as follows. Set
$\theta(\sigma)= \sigma$ if $\sigma$ is unmatched (i.e. if $\sigma$ is a critical simplex); set $\theta(\sigma)= \tau$ if $\sigma$ is matched with $\tau$ and $\dim \sigma < \dim \tau$.

A {\it discrete vector field} $V : K \rightarrow K$ is a collection of pairs $(\sigma, \tau)$ such that $\sigma$ is a codimension-one face of $\tau$, and no face of $K$ belongs to two different pairs of $V$. A {\it gradient path} in $V$ is a concatenation of pairs of $V$ $(\sigma_{0},\tau_{0}), (\sigma_{1},\tau_{1}), ..., (\sigma_{k},\tau_{k}), k \geq 1$ such that $\sigma_{i+1}$ differs from $\sigma_{i}$ and it is a codimension-one face of $\tau_{i}$. We say a gradient path is {\it closed} if $\sigma_{0} = \sigma_{k}$ for some $k$. A discrete vector field is a Morse matching if and only if $V$ contains no closed gradient paths (for the proof see \cite{forman_1998}, chapter $9$, page $131$).

An important result in discrete Morse theory is the following.
If $K$ is a simplicial complex with a discrete Morse function, then $K$ is homotopy equivalent to a CW complex with exactly one cell of dimension $i$
for each critical simplex of dimension $i$ (for the proof see \cite{forman_2002}, chapter $2$, page $10$).

Inside a simplicial complex $K$ we call a simplex ${\it free}$ if it is a face of a single simplex of $K$. An {\it elementary collapse} is the deletion of a free face $\sigma$ from a simplicial complex $K$. We say $K$ elementary collapses to $K \setminus \sigma$. We say $K$ {\it collapses} to $K_{1}$ if $K$ can be reduced to $K_{1}$ by a sequence of elementary collapses. We call a complex {\it collapsible} if it can be reduced by elementary collapses to one of its vertices.

The collapsibility of simplicial complexes can be studied using discrete Morse theory due to the following result. A simplicial complex is collapsible if and only if it admits a discrete Morse function with a single critical simplex (for the proof see \cite{forman_1998}).

Let $K$ be a locally finite simplicial complex. We say $K$ has an {\it arborescent structure} if it is a monotone union $\cup_{n=1}^{\infty}L_{n}$ of a sequence of collapsible subcomplexes $L_{n}$.

\section{The collapsibility of systolic simplicial complexes}

In this section we prove, using discrete Morse theory, that finite systolic simplicial complexes are collapsible. We shall fix a vertex and consider it as being a convex subcomplex of the complex. Our proof relies on the fact that in a systolic complex the projection of any simplex in the $1$-ball around a convex subcomplex of the complex is a single simplex.

Let $K$ be a systolic simplicial complex.
We start by defining a matching on the complex conveniently. Once the construction of the matching is done, our next goal will be to show that the chosen matching is, due to properties of systolic complexes, a Morse matching. The arrows will be drawn along projection rays onto a convex subcomplex of the complex.

We fix a vertex $v$ of $K$. Since any simplex in a flag complex is a $3$-convex subcomplex of $K$, we note that $Q = \{v\}$ is $3$-convex. Moreover, the fact that any $3$-convex subcomplex in a flag complex is locally $3$-convex implies that $Q$ is locally $3$-convex. So, because $Q$ is connected, it is a convex subcomplex of $K$.

We will consider projections of simplices onto balls around $Q$. We start by projecting simplices $\sigma$ of $K$ that belong to spheres $S_{n}Q$ around $Q, n \geq 2$. The matching of simplices in the $1$-sphere around $Q$ will be analyzed separately.

Let $\sigma$ be a simplex of $S_{n}Q, n \geq 2$. We consider the projection map $\overline{\pi}  : B_{n}Q \rightarrow B_{n-1}Q$. The projection of $\sigma$ onto $B_{n-1}Q$ is $\tau = \overline{\pi}_{B_{n-1}Q}\sigma$. We have two cases. Either $\sigma \cap B_{n-1}Q = \varnothing$ in which case let $\tau = \rm{St} (\sigma,K) \cap B_{n-1}Q$.
Or, else $\sigma \cap B_{n-1}Q \neq \varnothing$ in which case let $\tau = \sigma \cap B_{n-1}Q$.

In the first case, according to Lemma \ref{1.2.3}, the intersection of the star of $\sigma$ with $B_{n-1}Q$ is a unique simplex $\tau$. Moreover, $\sigma$ and $\tau$ are two consecutive simplices in the projection ray from $\sigma$ to $Q$. Hence they are disjoint and span a simplex of $K$.

Our goal is to match simplices with one another in order to construct gradient paths. Hence we must pair simplices with each other such that one simplex in the pair is a codimension-one face of the other. We borrow the idea from the proof given on CAT(0) spaces in \cite{benedetti_2012} (see chapter $3.1$). We choose a vertex $w$ of $\tau$. Since $\sigma$ and $\tau$ are disjoint and span a simplex, the vertex $w$ and the simplex $\sigma$ span another simplex, namely $\sigma \ast w$. Note that $\sigma$ is a codimension-one face of $\sigma \ast w$.
We define a pointer function $y_{Q} : K \rightarrow K$ which maps $\sigma$ (a simplex disjoint with $B_{n-1}Q$) into the vertex $w$. Note that $y_{Q}$ is not injective.

We analyze further the case when $\sigma$ has nonempty
intersection with $B_{n-1}Q$. Let $\tau = \sigma \cap B_{n-1}Q$.

Note that, since $\tau$ is a face of $\sigma$, any vertex of $\tau$ is a vertex of $\sigma$. Hence by choosing any vertex $u$ of $\tau$ and matching $\sigma$ with $\sigma \ast u,$ we would in fact match $\sigma$ with $\sigma$. So $\sigma$ would be a critical simplex of $K$. Our aim, however, is to construct the matching in such a manner that it has a single critical simplex (namely a critical vertex). We can therefore no longer argue as in the first case.
For those simplices $\sigma$ of $K$ which have nonempty intersection with $B_{n-1}Q$, we proceed as follows.
We choose a vertex $u_{1}$ of $\tau$. There is a unique directed geodesic $\tau_{1}\tau_{2} ... \tau_{n}$ joining $u_{1}$ with $v$. Note that $u_{1}$ is a vertex of both $\tau$ and $\tau_{1}$. The simplices $\tau$ and $\tau_{1}$ do not necessarily coincide. There exists a vertex $u_{2}$ of $\tau_{2}$ such that $d_{c}(u_{1},u_{2}) = 1$. Such a vertex $u_{2}$ exists because any sequence of vertices in a systolic complex, such that any two of its
consecutive vertices belong to two consecutive simplices in a directed geodesic, is a geodesic in the 1-skeleton of the complex. Note that the simplices $\tau_{1}$ and $\tau_{2}$ are disjoint since they are consecutive simplices of a directed geodesic.
Because $u_{1}$ is a vertex of $\tau$, knowing that $\tau$ is a face of $\sigma$, $u_{1}$ is also a vertex of $\sigma$. Hence, since $u_{2}$ is a neighbor of $u_{1}$,
the join $u_{2} \ast \sigma$ is a simplex of $K$.
For all simplices $\sigma$ of $K$ for whom $\sigma \cap B_{n-1}Q \neq \varnothing$, the pointer function $y_{Q}$ maps $\sigma$ into the vertex $u_{2}$.

Consider further simplices $\sigma$ of $K$ which belong to $S_{1}Q$. Note that $\rm{}St \sigma \cap Q = \tau$ whereas $\tau \subset Q$. But $Q = \{v\},$ and therefore $\tau = v$. So $y_{Q}(\sigma) = v$.

We will obtain a matching from $y_{Q}$ by pairing each still unmatched simplex $\sigma$ with the simplex $\sigma \ast y_{Q}(\sigma)$. If the pointer function maps distinct simplices $\sigma_{1}, \sigma_{2}$ to the same vertex, we may proceed since although $y_{Q}(\sigma_{1}) = y_{Q}(\sigma_{2}),$ we have $\sigma_{1} \ast y_{Q}(\sigma_{1}) \neq \sigma_{2} \ast y_{Q}(\sigma_{2}).$ Note that the pointer function might map distinct simplices to the same vertex $x$. This is not a problem however since we do not match $x$ to distinct simplices in the complex, but we match joins of $x$ with distinct simplices in the complex to these simplices.

Next we define a matching $\theta_{Q} : K \rightarrow K$ associated with  projection maps $\overline{\pi}  : B_{n}Q \rightarrow B_{n-1}Q$ onto balls around $Q$, $n \geq 1$. The definition is recursive on dimension. We start defining $\theta_{Q}$ for the faces of lower dimension and keep on matching simplices with one of their codimension-one faces until we reach facets (i.e. simplices whose dimension equals one less than the dimension of the complex).

We set $\theta_{Q}(\emptyset) = \emptyset$.
For any simplex $\sigma$ of $K$, if for all simplices $\tau \subsetneq \sigma$ one has $\theta_{Q}(\tau) \neq \sigma,$ we define $\theta_{Q}(\sigma) = y_{Q}(\sigma) \ast \sigma$ (i.e. if a simplex $\sigma$ has not been matched yet with any of its faces $\tau$, then match $\sigma$ with $y_{Q}(\sigma) \ast \sigma$).

Note that once we have performed the matching on all simplices of $K$, each simplex of $K$ is either in the image of $\theta_{Q}$, or in the domain of $\theta_{Q}$, or $\theta_{Q}(\sigma) = \sigma$ (i.e. $\sigma$ is a critical simplex).

\begin{definition}
Let $\theta : K \rightarrow K$  be a matching on the simplices of a simplicial complex $K$. We say $\theta$ is a {\it gradient matching} if $\theta = \theta_{Q}$ whereas $\theta_{Q} : K \rightarrow K$ is a matching associated with projection maps onto balls around $Q = \{v\}$.
\end{definition}

\begin{theorem}\label{2.1}
Let $K$ be a systolic simplicial complex and let $Q = \{v\}$ be a convex subcomplex of $K$. The induced gradient matching $\theta_{Q}$ is a Morse matching.
\end{theorem}

\begin{proof}
Given the matching performed above, we note that for any simplex $\sigma$ of $K$, the pairs $(\sigma, \theta_{Q}(\sigma))$ form a discrete vector field $V$. Moreover the discrete vector field $V$ contains no closed gradient paths.
This holds because
the gradient paths point along projection rays joining simplices in the complex with $v$. Projection rays coincide with directed geodesics. There is a unique directed geodesic from any vertex of the complex to $v$. Although there is another unique directed geodesic pointing from $v$ in the opposite direction, note that the gradient paths are defined along those directed geodesics pointing into $v$.
\end{proof}

\begin{corollary}\label{2.2}
Let $K$ be a systolic simplicial complex. Then $K$ is collapsible.
\end{corollary}

\begin{proof}
Fix a vertex $v$ of $K$. Note that $Q = \{v\}$ is a convex subcomplex of $K$. Perform on the face poset of $K$ the Morse matching constructed at the beginning of this section. According to the previous theorem, the matching $(\sigma, \theta_{Q}(\sigma))$ induces a Morse matching, i.e. a discrete vector field with no closed gradient paths. The only simplex matched to itself (i.e. it is critical) is the vertex $v$. So $K$ admits a discrete Morse function with a single critical simplex. $K$ is therefore collapsible.
\end{proof}

\section{The combinatorial structure of locally finite systolic simplicial complexes}

In this section we extend the result obtained by Corson and Trace in dimension $2$ regarding the arborescent structure of locally finite systolic simplicial complexes, to all dimensions (see \cite{corson_1998}, Theorem $3.5$).

The barrier of dimension $2$ may be broken due to the following result proven in \cite{lazar_2013} (Corollary $3.4$) and \cite{chepoi_2009} (Theorem $5.1$).

\begin{theorem}\label{3.1}
Any finite systolic simplicial complex is collapsible.
\end{theorem}

The main result of the section is the following.

\begin{theorem}\label{3.2}
Let $K$ be a locally finite systolic simplicial complex.
Then $K$ is a monotone union of a sequence of collapsible subcomplexes.
\end{theorem}

\begin{proofpar}\label{3.3}

We fix a vertex $v$ of $K$. For each integer $n$, let $B_{n} := B_{n}(v,K)$ denote the full subcomplex of $K$ generated by the vertices that can be joined to $v$ by an edge-path of length at most $n$. Note that $\{v\}$ is a convex subcomplex of $K$.
Note that for each $n$, $B_{n}$ is a combinatorial ball in a systolic complex. So, since $\{v\}$ is a convex subcomplex in a systolic complex, each such $B_{n}$ is itself a finite systolic complex. Hence for each $n$, $B_{n}$ is, according to Theorem \ref{3.1}, collapsible.
$K$ is therefore the monotone union $\cup_{n=1}^{\infty}B_{n}$  of a sequence of collapsible subcomplexes. Thus $K$ has an arborescent structure.
\qed
\end{proofpar}

\section{The combinatorial structure of locally finite CAT(0) simplicial complexes}

In this section we investigate the combinatorial structure of locally finite simplicial complexes that are CAT(0) for the standard piecewise Euclidean metric.

The section's main result relies on the following theorem proven in \cite{benedetti_2012} (Corollary $3.2.4$).

\begin{theorem}\label{2.1}
Let $K$ be a finite simplicial complex that is a CAT(0) space for the standard piecewise Euclidean metric.
Then $K$ is collapsible.
\end{theorem}

In dimension $2$, however, according to \cite{lazar_2010_8} (see chapter $3$, page $35$), all CAT(0) simplicial coplexes are collapsible, not only those endowed with the standard piecewise Euclidean metric.
Still, our proof for the fact that locally finite CAT(0) simplicial complexes possess an arborescent structure works, even in the $2$-dimensional case, only for simplicial complexes with unitary edge lengths.

\begin{theorem}\label{2.3}
Let $K$ be a locally finite simplicial complex that is a CAT(0) space for the standard piecewise Euclidean metric.
Then $K$ is a monotone union of a sequence of collapsible subcomplexes.
\end{theorem}

\begin{proofpar}\label{2.4}
We fix a vertex $v$ of $K$. For each integer $n$, let $B_{n}$ be the full subcomplex of $K$ generated by the vertices that can be joined to $v$ by an edge-path of length at most $n$. Note that for each $n$, $B_{n}$ is a ball in a CAT(0) space. It is therefore contractible and in particular it is simply connected.

Furthermore, because $K$ is a CAT(0) space, it is locally a CAT(0) space. For each $n > 0$, $B_{n}$ is therefore itself  locally a CAT(0) space. Because each $B_{n}, n > 0$ is simply connected and locally a CAT(0) space, we may conclude that each such ball is a CAT(0) space.  Hence for each $n$, $B_{n}$ is a finite simplicial complex that is a CAT(0) space for the standard piecewise Euclidean metric. So, according to Theorem \ref{2.1}, each such $B_{n}$ is collapsible.
$K$ is therefore the monotone union $\cup_{n=1}^{\infty}B_{n}$ of a sequence of collapsible subcomplexes. This ensures that $K$ has an arborescent structure.
\qed
\end{proofpar}

\section{Final remarks}

The first collapsibility criterion involving CAT(0) metrics was given in $2008$ by
Crowley for complexes of dimension at most $3$ endowed with standard piecewise Euclidean metrics. Her result was extended to all dimensions by Adiprasito and Benedetti in \cite{benedetti_2012} both for simplicial and polyhedral complexes. Both results require additional hypotheses and follow by applying discrete Morse theory.

Another method for showing the collapsibility of CAT(0) simplicial complexes is given in dimension $2$ in \cite{lazar_2010_8} (see chapter $3.1$, page $36$). This second method uses the definition of an elemenatry collapse and it relies on the fact that, according to White (see \cite{white_1967}), strongly convex simplicial $2$-complexes have a simplex with a free face. In higher dimensions, however, it is still an open question whether the same method ensures the collapsibility of CAT(0) complexes. We assume, however, that White's result can be extended to arbitrary dimension: finite strongly convex simplicial complexes should also have a facet with a free face. Once this problem is solved, we will have to find the new geodesic segments in the subcomplex obtained by performing an elementary collapse on the CAT(0) complex. In  dimension $2$ there exists a unique geodesic segment joining any two points in the subcomplex (see \cite{lazar_2010_8}, chapter $3.1$, page $36$). In higher dimensions, however, this is no longer true for any two points in the complex. Namely, the problem is with those points joined in the CAT(0) complex by a unique segment passing through the simplex one deletes. This segment is no longer necessarily unique in the subcomplex obtained by performing the elementary collapse. We hope the method works at least for CAT(0) complexes endowed with standard piecewise Euclidean metrics. By fixing a vertex of the complex, considering combinatorial balls around the fixed vertex and performing as many elementary collapses on the complex as necessary, we should reach another ball of smaller radius around the same vertex and recover the CAT(0) metric the complex was initially endowed with. Note that each such ball is a CAT(0) space. Once we have recovered the CAT(0) metric, since the strongly convex metric on the subcomplex ensures the existence of the simplex with the free face, we may perform another elementary collapse.

Another notion related to collapsibility is shellability (see \cite{ziegler_1998}). We call a pure $n$-dimensional simplicial complex $K$ shellable if it is the union of a shellable $n-$complex $K_{1}$ and an $n$-simplex $K_{2}$ such that $K_{1} \cap K_{2}$ is a shellable $(n-1)-$complex. Intuitively, shellable complexes can be assembled one $n$-face at a time. This induces a total order called shelling order on the facets of the complex. Since one may collapse away all top-dimensional faces in the reverse of the shelling order, shellable contractible complexes are collapsible.
Systolic complexes being both contractible and collapsible, one may naturally ask whether they are also shellable.
It turns out they are not. This is the case since, according to Zeeman, every shellable $n$-dimensional pseudomanifold is either an $n$-ball or a $n$-sphere. A pseudomanifold is a pure $n$-complex in which every $(n-1)$-face belongs to either one or two $n$-faces, but not more. Every triangulation of a manifold is a pseudomanifold but not the other way round.
No triangulation of a $2$-sphere, however, is $6$-large and hence no triangulation of a manifold of dimension at least $3$ is systolic (because such manifold has $2$-spherical links). One may hence conclude that systolic complexes are not shellable.

\end{document}